\documentclass[12pt]{amsart}

\usepackage{amsmath,amssymb,amscd}
\usepackage{multicol}
\usepackage{graphicx}
\usepackage{color}

\usepackage{amsfonts}
\usepackage{amsthm}
\usepackage{amsmath}
\usepackage{amscd}
\usepackage{amssymb}
\usepackage{enumerate}
\usepackage{ulem}
\usepackage{eufrak}
\usepackage{eucal}
\usepackage{fontenc}
\usepackage[all]{xy}
\usepackage{array}
\usepackage{mathrsfs}
\usepackage{graphicx}
\usepackage[all]{xy}
\usepackage{setspace}
\usepackage{anysize}
\marginsize{1.5in}{1in}{1in}{1.5in}
	
\def\ring#1{\ifmmode \mathaccent'027 #1\else \rm\accent'027 #1\fi}

\newcommand{\RR}{\mathbb{R}}
\newcommand{\CC}{\mathbb{C}}

\newcommand{\ZZ}{\mathbb{Z}}

\newcommand{\FFF}{\mathcal{F}}

\newcommand{\id}{\mathrm{id}}

\newcommand{\rel}{\mathrm{rel \;}}
\newcommand{\Diff}{\mathrm{Diff}}

\newcommand{\noin}[1]{[#1, G/Top]}

\newcommand{\OO}[1]{\mathrm{O}(#1)}
\newcommand{\SO}[1]{\mathrm{SO}(#1)}

\newcommand{\tinycirc}{{^{{\,}_{\text{\rm o}}}}}

\newtheorem{thm}{Theorem}

\newtheorem*{claima}{Claim}

\newtheorem{asrt}{Assertion}
\newtheorem{obs}[asrt]{Observation}

\newtheorem{hyp}[asrt]{Hypothesis}
\newtheorem{lemma}[asrt]{Lemma}
\newtheorem{prop}[asrt]{Proposition}

\numberwithin{asrt}{section} 

\begin{document}
\setcounter{page}{1}

\subjclass[2000]{Primary: 57R67, 57S05, 57S17. Secondary: 19J25, 57S25.}

\title{Finite Symmetries of $S^4$}
\author{Weimin Chen, S{\l}awomir Kwasik and Reinhard Schultz}
\maketitle

\begin{abstract}
This paper discusses topological and locally linear actions of finite groups on $S^4$. 
Local linearity of the orientation preserving actions on $S^4$ forces 
the group to be a subgroup of $\SO{5}$. On the other hand, orientation reversing topological 
actions of ``exotic" groups $G$ ({\it i.e.} $G\not\subset \OO{5}$) on $S^4$ are constructed, and
local linearity and stable smoothability of the actions are studied. 
\end{abstract}

\vspace{1cm}

In the theory of transformation groups, one class of closed manifolds plays a 
special role;  namely, standard spheres $S^{n}$.  Classical results describe all finite 
group actions on $S^1$ and $S^2$; all such actions are conjugate to the linear ones.

The existence of wild topological actions on $S^3$ ({\it cf.} \cite{bredon}) suggests restricting 
attention to group actions which are more regular. Consequently, one usually requires the
action to be smooth or more generally locally linear ({\it cf.} \cite{bredon}). If one restricts to 
locally linear, say orientation preserving, actions on $S^3$ (which are smoothable by 
\cite{KwLee}), then again all these actions are topologically conjugate to linear actions ({\it cf.} 
\cite{DL}, \cite{MT2}). In particular, if $G$ is a finite group acting locally linearly,
orientation-preservingly (and effectively) on $S^3$, then $G$ is isomorphic to a subgroup of $\SO{4}$. 
This is a highly nontrivial result involving 
Perelman's work on the Thurston's Geometrization Conjecture ({\it cf.} \cite{BLP}, \cite{MT2}). Strangely enough the fact that $G\subset \SO{4}$, as far as we know, does not have a direct proof (i.e. without using the linearity of the action). It should be pointed out that a locally linear action of $G$ on $S^3$ is also 
topologically conjugate to a linear action in the orientation reversing case ({\it cf.} \cite{DL}). In this case 
$G\subset \OO{4}$, but the main difficulties are related to the orientation preserving case which has attracted most of the attention ({\it cf.} \cite{SmithConj}).

For $S^4$ there are smooth, orientation preserving actions of $G=\ZZ_p$ which are not 
conjugate to linear ones (not even topologically). The existence of nonlinear finite group 
actions on $S^4$ still left
open possibility that these groups are subgroups of the corresponding orthogonal groups. 
In particular, there are two appealing conjectures ({\it cf.} \cite{MZ}, \cite{Z2005}, \cite{MZ0}). 

\vspace{3mm}

\noindent{\bf Conjecture 1.} 
{\it If a finite group $G$ acts smoothly (or locally linearly) and orientation preservingly on
$S^4$, then $G$ is isomorphic to a subgroup of $\SO{5}$.
}

\vspace{3mm}

\noindent{\bf Conjecture 2.} 
{\it If $G$ acts orientation reversingly on $S^4$, then $G$ is isomorphic to a subgroup of 
$\OO{5}$.
}

Significant progress toward the proof of Conjecture 1 was made in \cite{MZ}. Our first main result is
the completion of the proof of Conjecture 1. In particular, the following holds. 

\begin{thm}\label{thm:A}
If a finite group $G$ acts locally linearly and orientation 
preservingly on a homology $4$-sphere $\Sigma$, then $G$ is isomorphic 
to a subgroup of $\SO{5}$. 
\end{thm}

It is however worthwhile to note the following:

\vspace{2mm}

\noindent{\bf Observation}. {\it There is a finite group $\pi$ acting topologically and orientation-preservingly on $S^4$
(semifreely with two fixed points) such that $\pi$ is not isomorphic to a subgroup of $\SO{5}$. 
}
\vspace{2mm}

Our remaining results concern Conjecture 2. In this direction we have the following: 

\vspace{2mm}

\noindent{\bf Fact.}
{\it There is a finite group $\pi$ acting topologically and orientation reversingly on $S^4$ without 
fixed points such that $\pi$ is not isomorphic to a subgroup of $\OO{5}$.
}
\vspace{2mm}

The above fact gives a topological solution (in negative) to Conjecture $2$. The locally linear (or smooth) case of
Conjecture $2$ is 
still open. To be more precise,  we start with a construction of topological,
orientation reversing, pseudo-free actions of groups $Q(8p,q)$, for certain pairs of $(p,q)$,
on $S^4$. Note that $Q(8p,q)\not\subset \OO{5}$. Then we describe an approach for 
constructing potential counterexamples to the locally linear version of Conjecture 2 and then
discuss stable smoothability of these potential locally linear actions. This approach uses computational techniques involving surgery groups ({\it cf.} \cite{W2}, \cite{O}, \cite{Ma}) and algebraic number theory. It also leads to questions concerning free finite group actions on 
homology $3$-spheres which are of independent interest. Our results in this direction are summarized in the following:

\begin{thm}\label{thm:B}
There exist orientation reversing (pseudo-free) actions of groups $Q(8p,q)$, for some 
pairs $(p,q)$, on $S^4$. Suppose such an action $\Phi$ is locally linear. If 
$\Phi\times\mathrm{id}_\RR$ is not equivariantly smoothable, then there is also an 
action $\Phi^\prime$ of the same type 
such that $\Phi^\prime\times\mathrm{id}_\RR$ is equivariantly smoothable, 
and likewise if $\Phi\times\mathrm{id}_\RR$ {\sf is} equivariantly smoothable, 
then there is a locally linear action $\Phi^\prime$ of the same type such that for all
nonnegative integers $k$, the product action
$\Phi^\prime\times\mathrm{id}_{\RR^k}$ is {\sf not} equivariantly smoothable.
\end{thm}

In other words, the potential locally linear actions of $Q(8p,q)$ on $S^4$ come in twin
pairs, such that one is stably equivariantly smoothable but the other is not. 

In view of Theorem 2, we would like to advertise the smooth version of Conjecture 2,  stated
here as 

\vspace{3mm}

\noindent{\bf Problem 1.}
{\it  Let $G$ be a finite group acting smoothly and orientation reversingly on $S^4$. Is $G$ 
isomorphic to a subgroup of $\OO{5}$?
}
\vspace{3mm}

Our interest in the questions of this paper was also partly motivated by the paper of C. B.
Thomas \cite{Th}.  Namely, in \cite{Th} the existence
of exotic smooth actions of finite groups 
was used to show that the natural inclusions of $\OO{6}$ in the 
automorphism groups $Top(S^5)$ and ${\Diff}(S^5)$ are not homotopy
equivalences.  The exotic actions used by Thomas are actually by subgroups of $\SO{6}$. 
In particular these actions include the 
action first constructed by T. Petrie for the metacyclic group of order $21$, {\it cf.} \cite{Th}.
The constructions of this paper would then suggest 
that the inclusion of $\OO{5}$ in $Top(S^4)$ is also not a homotopy
equivalence, for the group actions constructed in this paper are by ``exotic" groups which
act on $S^4$ but obviously can not act linearly on $S^4$ {\sf because they are not 
subgroups of $\OO{5}$}.

In fact, one can use homotopy theoretic methods as in \cite{RS}
to show that the group $\pi_3(\,Top(S^n)\,)$ contains a copy of 
$\ZZ_2$ if $n\geq 4$, which immediately implies that 
the natural inclusion of $\OO{5}$ in $Top(S^4)$ is not a homotopy
equivalence.  On the other hand, if $m\leq 3$ the inclusions 
$\OO{m+1}\subset \Diff(S^m)\subset Top(S^m)$ {\sf are} homotopy 
equivalences; the case $m=3$ is the celebrated result of A. Hatcher 
in \cite{H}.  This means that only the following question is open:

\vspace{2mm}

\noindent\textbf{Four-dimensional Smale Conjecture. }
{\sl The natural inclusion of $\OO{5}$ in $\Diff(S^4)$ is a homotopy
equivalence.}

\vspace{2mm}

This question provides additional motivation for asking whether the 
actions in Theorem 2 are in fact equivariantly smoothable ({\it cf.} Problem 1).  
One possible approach to study the smoothability question could involve 
the use of gauge theory, possibly combined with some other techniques; 
we hope to return to such problems in a future paper.

It would also be interesting to know if the group actions constructed
in this paper can be used to provide a different proof that
$\OO{5}\not\simeq Top(S^4)$ in the spirit of \cite{Th}.

The homotopy equivalence $\OO{4}\subset Top(S^3)$ makes the following problem quite interesting.

\vspace{2mm}

\noindent{\bf Problem 2.}
{\it If a finite group acts effectively and topologically on $S^3$, must it be isomorphic to a subgroup 
of $\OO{4}$?
}

\vspace{2mm}

Finally, in view of Theorem 1, the following problem becomes of considerable interest.

\vspace{2mm}

\noindent{\bf Problem 3.}
{\it What is the smallest dimension $n$ of the sphere $S^n$ such that $S^n$ admits an
orientation preserving smooth  (locally linear) action of a finite group $G$ with $G$ not
being isomorphic to a subgroup of $\SO{n+1}$?
}

\vspace{3mm}

The paper is divided into two parts. The first contains the proof of Theorem 1. The second
contains the proof of Theorem 2.
Our proofs combine quite intricate considerations involving the theory of finite groups and
surgery theoretic considerations based on some computational results for Wall's surgery
obstruction groups \cite{W1}, \cite{CS2}.

\section{Orientation preserving actions on $S^4$}

\noindent{\it Proof of Theorem 1}. We start our proof by recalling some arguments from
\cite{MZ0, MZ}. 

Theorem 1 was verified by Mecchia and Zimmermann ({\it cf.} \cite{MZ0, MZ}) in many special 
cases; {\it e.g.}, when $G$ is nonsolvable. In fact, they showed that $G$ must be isomorphic to 
a subgroup of $\SO{5}$, except for the following indeterminacy of index two in the case where 
$G$ is solvable. 

{\it There is an index two subgroup $H$ of $G$ which is isomorphic to a subgroup of $\SO{4}$ 
such that $\Sigma^H=\{p_{+},p_{-}\}$, and $\Sigma^G=\emptyset$. 
}

We will prove Theorem 1 by removing the above indeterminacy of $G$, showing that
$G$ is isomorphic to a subgroup of $\OO{4}$. For reader's convenience, we shall give an 
outline of
Mecchia and Zimmermann's arguments in \cite{MZ} first. Let $E$ be the maximal semisimple 
normal subgroup
of $G$ (recall that a semisimple group is a central product of quasisimple groups, where a quasisimple
group is a perfect central extension of a simple group, {\it cf.} \cite{Su}, Chapter 6.6.) Then the arguments
in \cite{MZ} are divided into two cases: (1) $E$ is nontrivial, (2) $E$ is trivial. 

In an earlier paper \cite{MZ0}, Mecchia and Zimmermann showed that if $E$ is nontrivial, then $E$ 
must be isomorphic to one of the following groups: 
$A_5,A_6, A_5^\ast, A_5^\ast\times_{\ZZ_2} A_5^\ast$,
where $A_5^\ast$ denotes the binary dodecahedral group. With this understood, the arguments for
case (1) proceed by considering the subgroup $\tilde{E}$ which is generated by $E$ and its 
centralizer in $G$, exploiting the fact that $G/\tilde{E}$ is isomorphic to a subgroup of the outer
automorphism group $\text{Out }(E)$ which is small for the possible groups $E$. 

The index two indeterminacy occurs in case (2) where $E$ is trivial. In this case, one considers instead
the Fitting subgroup $F$ of $G$ (recall that the Fitting subgroup is the maximal nilpotent normal subgroup).
The subgroup $F^\ast$ generated by $F$ and $E$ is called the generalized Fitting subgroup of $G$, and
as its main property, $G/F^\ast$ is isomorphic to $\text{Out }(F^\ast)$. Since $E$ is trivial, one has in this
case the crucial fact that $G/F$ is isomorphic to a subgroup of $\text{Out }(F)$. The Fitting subgroup
$F$ is a direct product of its Sylow p-subgroups, for different primes $p$, and each of the Sylow p-subgroups of $F$ is normal in $G$. With this understood, the aforementioned index two indeterminacy in \cite{MZ} occurs precisely in the following situation: 

Let $S$ be any Sylow p-subgroup of $F$, and let $Z$ be the maximal elementary Abelian subgroup
in the center of $S$ (note that $Z$ is normal in $G$). Then one of the following is true ({\it cf.} \cite{MZ},
pp. 745-746).

\begin{itemize}
\item [{(i)}] $Z$ has rank one and $\Sigma^Z=\{p_{+},p_{-}\}$.
\item [{(ii)}] $Z$ has rank two and $\Sigma^Z=\{p_{+},p_{-}\}$.
\item [{(iii)}] $Z$ has rank three and $\Sigma^Z=\{p_{+},p_{-}\}$, with $p=2$. 
\end{itemize}

With the preceding understood, our arguments for removing the index two indeterminacy in \cite{MZ}
are based on the following Atiyah-Bott Theorem ({\it cf.} \cite{W1}).

\vspace{2mm}

{\bf Atiyah-Bott Theorem}: {\it Let a finite group $\Gamma$ act locally linearly, semifreely, and orientation 
preservingly on $\Sigma$ with two isolated fixed points $\{p_{+},p_{-}\}$. Then there is an orientation
reversing isomorphism between $T_{p_{+}}\Sigma$ and $T_{p_{-}}\Sigma$, under which the 
corresponding representations of $\Gamma$ become equivalent.}

\vspace{2mm}

The Atiyah-Bott Theorem is a crucial ingredient for proving that in the case of {\sf orientation--preserving} actions of $G$ on $\Sigma$, the aforementioned index two subgroup $H$ of $G$ has certain special properties. Namely, we have the following key lemma.

\begin{lemma}\label{lem:1}
Let $h\in H$ be any element such that $h^2\neq 1$ and for any $k\in\ZZ$ with $h^k\neq 1$, 
$\Sigma^{h^k}=\{p_{+},p_{-}\}$. 
Then for any $\lambda\in G\setminus H$, $\lambda$ does not lie in the set
$$C_G^\ast(h)=\{x\in G| xhx^{-1}=h \mbox{ or } h^{-1}\}.$$
\end{lemma}

\begin{proof}
Assume $\lambda\in C_G^\ast(h)$. We shall derive a contradiction.

Consider first the case where $\lambda\in C_G(h)$, the centralizer of $h$. In this case, 
the fixed-point set 
$\Sigma^\lambda$, which is either a $0$-sphere or $2$-sphere, is invariant under $h$. 
Furthermore, it is disjoint from $\Sigma^h=\{p_{+},p_{-}\}$ because $\lambda$ switches the 
two points. But this is a contradiction, for either $h^2\neq 1$ fixes $\Sigma^\lambda$ or $h$ 
has a fixed point on it.

Consider next the case where $\lambda\in C_G^\ast(h)\setminus C_G(h)$. By the Atiyah-Bott
Theorem, there is an orientation reversing isomorphism between $T_{p_{+}}\Sigma$ 
and $T_{p_{-}}\Sigma$, under which the corresponding action of $h$ becomes equivalent.
We will denote this real $4$-dimensional $\langle h\rangle$-representation by $V$. Note that 
$\lambda:V\rightarrow V$ is {\sf orientation reversing}, and for any $v\in V$, 
$h^{-1}\lambda(v)=\lambda h(v)$. We will show these two conditions contradict each other.

First, a digression. Let $A$ be a $2\times 2$ matrix of real coefficients and let 
$$
H_\theta=\left (\begin{array}{ll}
\cos \theta & -\sin\theta\\
\sin \theta & \cos\theta\\
\end{array}
\right )
$$
where $\sin\theta\neq 0$, such that $H_\theta^{-1} A=AH_\theta$. Then it's an elementary 
exercise to see that there are $a,a^\prime\in \RR$, such that 
$$
A=\left (\begin{array}{ll}
a & a^\prime\\
a^\prime & -a\\
\end{array}
\right )
$$
In particular, $\det A=-a^2-(a^\prime)^2\leq 0$ for any such a matrix $A$, and $\det A=0$
 if and only if $A=0$.

Now suppose $V=V_1\oplus V_2$ where $V_1,V_2$ are 2-dimensional, non-equivalent 
$\langle h\rangle$-representations. Then $\lambda(V_i)=V_i$, $i=1,2$, because 
$\lambda(V_i)$ is equivalent to $V_i$ for each $i$, as $\lambda\in C_G^\ast(h)$ by assumption. 
Then $h^{-1}\lambda(v)=\lambda h(v)$ implies that the restriction of $\lambda$ on each $V_i$
may be represented by a $2\times 2$ matrix $A_i$ satisfying 
$H_{\theta_i}^{-1} A_i=A_iH_{\theta_i}$ for some $\theta_i$. Note that $\Sigma^{h^2}=\{p_{+},
p_{-}\}$ implies that $\sin \theta_i\neq 0$ for each $i$, and consequently, $\det A_i\leq 0$.
But this implies that $\lambda: V\rightarrow V$ is orientation preserving, which is a contradiction.

Suppose $V=V_1\oplus V_2$ where $V_1,V_2$ are 2-dimensional $\langle h\rangle$-representations which are equivalent. Without loss of generality, we assume, after giving proper orientation to each $V_i$, that the action of $h$ on each $V_i$ is given by a rotation by an 
angle $\theta$, where $0<\theta<\pi$ (in particular, $\sin\theta\neq 0$). If we represent 
$\lambda$ by a matrix $A=(A_{ij})$, where each $A_{ij}$ is a $2\times 2$ matrix, then
$h^{-1}\lambda(v)=\lambda h(v)$ implies that $H_\theta^{-1} A_{ij}=A_{ij}H_\theta$ for
each $i,j$. Without loss of generality, assume $A_{11}$ is non-singular. Then 
$$
\det A=\det A_{11}\cdot \det(A_{22}-A_{21}A^{-1}_{11}A_{12})>0,
$$
because $A_{22}-A_{21}A^{-1}_{11}A_{12}$ satisfies the equation 
$H_\theta^{-1} A=AH_\theta$ also, and hence has non-positive determinant. But this
contradicts the fact that $\lambda:V\rightarrow V$ is orientation reversing. Hence
Lemma \ref{lem:1} follows.

\end{proof}

Given the above, we shall divide the proof of Theorem 1 into two stages. In Stage 1, 
we assume that there is a prime $p$ such that either (ii) or (iii) is true. 

\begin{prop}\label{prop:1}
Suppose there is a prime $p$ such that either (ii) or (iii) is true. Then $G$ is isomorphic 
to a subgroup of $\SO{4}$.
\end{prop}

The proof of Proposition 1.2 is based on the following observation.

\begin{obs}\label{obs:2}
Suppose $\lambda^2=1$ for any $\lambda\in G\setminus H$. Then $G$ is isomorphic 
to a subgroup of $\SO{4}$.
\end{obs}

\begin{proof}
We pick a $t\in G\setminus H$, and observe that $G\setminus H=\{ht|h\in H\}$. 
Then $\lambda^2=1$ for any $\lambda\in G\setminus H$ gives $t^2=1$, $htht=1$, 
which implies that
$$
h^{-1}=tht^{-1}, \;\; \forall h\in H.
$$
Furthermore, for any $h_1,h_2\in H$,
$$
h_1h_2=(th_1^{-1}t^{-1})\cdot (th_2^{-1}t^{-1})=t(h_2h_1)^{-1}t^{-1}=h_2h_1.
$$
Consequently, $H$ is Abelian, and $G$ has the following presentation
$$
G=\{h\in H,t| t^2=1, tht^{-1}=h^{-1}, \forall h\in H\}.
$$

To see that $G$ is isomorphic to a subgroup of $\SO{4}$, note that since $H\subset \SO{4}$ 
is Abelian, there is an identification $\RR^4=\CC^2$ such that $H$ can be represented by
complex $2\times 2$ matrices which are unitary and symmetric. 
With this understood, we extend the faithful representation of $H$ on $\CC^2$ to a faithful 
representation of $G$ on $\RR^4$ by defining the action of $t$ as follows:
$$
t\cdot (z_1,z_2)=(\bar{z}_1,\bar{z}_2), \;\; (z_1,z_2)\in\CC^2.
$$
This indeed is compatible with the action of $H$, as for any $h\in H$,
$$
tht^{-1}(z_1,z_2)=th(\bar{z}_1,\bar{z}_2)=\bar{h} (z_1,z_2)=h^{-1} (z_1,z_2).
$$
\end{proof}

Given the above observation, Proposition \ref{prop:1} follows from Lemma \ref{lem:2} 
and Lemma \ref{lem:3} below, 
which deal with (ii) and (iii) respectively. But first of all, the following fact, which follows 
from a simple argument in Smith theory, will be frequently used:

{\it Let $\lambda\in G\setminus H$ such that $\lambda^2\neq 1$. Then $\Sigma^\lambda$ is a
$0$-sphere disjoint from $\{p_{+},p_{-}\}$ and $\Sigma^{\lambda^2}$ is a $2$-sphere 
containing $\{p_{+},p_{-}\}$.
}

\begin{lemma}\label{lem:2}
Suppose there is a Sylow p-subgroup $S$ of the Fitting subgroup $F$ such that the maximal
elementary Abelian subgroup $Z$ of $S$ has rank two with $\Sigma^Z=\{p_{+},p_{-}\}$, then
$\lambda^2=1$ for any $\lambda\in G\setminus H$.
\end{lemma}

\begin{proof}
First, as shown in \cite{MZ0}, Lemma 4.1, 
there are exactly two index $p$ subgroups $Z_1,Z_2$ of $Z$ 
whose fixed-point set is a $2$-sphere, and for all other index $p$ subgroups of $Z$ the fixed-point
set is a $0$-sphere. For convenience we sketch its proof here. The action of $Z$ on $\Sigma$
obeys the following Borel Formula ({\it cf.} \cite{bredon})
$$
4-r=\Sigma_{i} (n(Z_i)-r)
$$
where $Z_i$ is running over all index $p$ subgroups of $Z$, $n(Z_i)$ denotes the dimension of
$\Sigma^{Z_i}$, and $r$ is the dimension of $\Sigma^Z$. Since each $Z_i$ is cyclic, $n(Z_i)$
is either $0$ or $2$ by Smith theory. The claim follows easily from Borel Formula with $r=0$. 

Let $\lambda\in G\setminus H$ and assume $\lambda^2\neq 1$. The action of $\lambda$ on
$Z$ by conjugation will leave the set $\{Z_1,Z_2\}$ invariant; in particular, $\lambda^2$ leaves
each of $Z_1$ and $Z_2$ invariant under conjugation. Consequently, $\Sigma^{Z_i}$ ($i=1,2$)
is invariant under $\lambda^2$. Now observe that the three $2$-spheres $\Sigma^{Z_1}$,
$\Sigma^{Z_2}$ and $\Sigma^{\lambda^2}$ all contain $\{p_{+},p_{-}\}$. The tangent plane
of $\Sigma^{\lambda^2}$ at $p_{+}$ is fixed under $\lambda^2$, and the tangent planes
of $\Sigma^{Z_1}$ and $\Sigma^{Z_2}$ at $p_{+}$ are distinct. It follows that $\Sigma^{\lambda^2}$
must coincide with one of $\Sigma^{Z_1}$, $\Sigma^{Z_2}$, say $\Sigma^{Z_1}$.
This implies that $\Sigma^\lambda\subset \Sigma^{\lambda^2}=\Sigma^{Z_1}$, and $\Sigma^{Z_1}$
is invariant under $\lambda$. It follows then that each of $Z_1$ and $Z_2$ is invariant under
the action of $\lambda$. Particularly, $\Sigma^{Z_2}$ is invariant under $\lambda$. Now observe
that $\Sigma^{Z_1}\cap \Sigma^{Z_2}=\Sigma^Z=\{p_{+},p_{-}\}$, so that the action of
$\lambda$ on $\Sigma^{Z_2}$ has no fixed points. This implies that $\lambda^2$ fixes 
$\Sigma^{Z_2}$ also, which is a contradiction. Hence $\lambda^2=1$.
\end{proof}

\begin{lemma}\label{lem:3}
Suppose the maximal elementary Abelian subgroup $Z$ of the Sylow 2-subgroup $S$ of $F$ 
has rank three with $\Sigma^Z=\{p_{+},p_{-}\}$. Then $\lambda^2=1$ for any $\lambda\in 
G\setminus H$. 
\end{lemma}

\begin{proof}
There are seven subgroups of index $2$ in $Z$. As shown in \cite{MZ0}, Lemma 4.1, the
Borel Formula implies that four of the subgroups have $1$-dimensional fixed-point set,
which is a $1$-sphere, and three subgroups have $0$-dimensional fixed-point set. 
Equivalently, of the seven involutions of $Z$, six of them has fixed-point set a $2$-sphere,
and exactly one involution has fixed point set a $0$-sphere, which is given by 
$\Sigma^Z=\{p_{+},p_{-}\}$. 

Denote by $h\in Z$ the involution with $\Sigma^h=\Sigma^Z=\{p_{+},p_{-}\}$. We denote the action
of any $g\in G$ on $Z$ given by conjugation by $g(x)=gxg^{-1}$, $\forall x\in Z$. Then clearly 
$g(h)=h$, $\forall g\in G$. 

We first prove that either $\lambda^2=1$ 
or $\lambda^6=1$ for any $\lambda\in G\setminus H$.
Let $Z_0=Z\setminus\{1,h\}$ be the set of the six involutions whose fixed-point set is
a $2$-sphere. Let $\lambda\in G\setminus H$ such that $\lambda^2\neq 1$.  
We shall examine the action of $\lambda$ on $Z_0$.

Two observations first: (a) if $x\in Z_0$ is fixed under $\lambda$, i.e., $\lambda(x)=x$,
so is $xh\in Z_0$, as $\lambda(xh)=\lambda(x)\lambda(h)=xh$; (b) for any $x\in Z_0$,
if $xh=\lambda^k(x)$ is in the orbit of $x$, then $\lambda^{2k}(x)=\lambda^k (xh)=\lambda^k(x)\lambda^k(h)=xhh=x$. 

With this understood, we first show that $\lambda$ can not act transitively on $Z_0$. This is
because for any $x\in Z_0$, (b) implies that $xh=\lambda^3(x)$. It follows easily
that the orbit $x,\lambda(x),\cdots, \lambda^5(x)$ has the following form
$$
x,y,z, xh,yh,zh,
$$
where $z=xy$ or $xyh$. Assume first that $z=xy$. Then the subgroup generated by $x,y,z$
is transformed under $\lambda$ to the subgroup generated by $y,z,xh$. However, this is
a contradiction because the two subgroups have different ranks. In the case where $z=xyh$,
the subgroup generated by $zh,x,y$ is transformed under $\lambda$ to the subgroup
generated by $x,y,z$. The former has rank two but the latter has rank three, which is also
a contradiction. Hence the claim follows. 

Next, it follows easily from (a) that the action of $\lambda$ on $Z_0$ can not have an orbit 
consisting of five elements. Now consider the next case where the action of $\lambda$ has 
an orbit consisting of four elements, say $x_1,x_2,x_3,x_4$. Let $\{x,y\}=Z_0\setminus\{x_i\}$ 
be the complement. We claim $y=xh$. To see this, note that if $x$ is fixed under $\lambda$, 
then so is $xh$,
which implies $xh\in Z_0\setminus\{x_i\}$. Hence $y=xh$. If $x$ is not fixed under $\lambda$,
then $y=\lambda(x)$, $x=\lambda(y)$, so that $xy$ is fixed under $\lambda$. This implies 
that $xy=h$, which is equivalent to $y=xh$. With this understood, note that the subgroup
generated by $x,y$ is invariant under $\lambda$, which has rank two, and furthermore,
note that its fixed-point set is $\{p_{+},p_{-}\}$, and that exactly
two subgroups of index $2$ has fixed-point
set a $2$-sphere. The argument of the previous lemma then shows that $\lambda^2=1$ must be
true, which is a contradiction. This proves the claim that the action of $\lambda$ can not have 
an orbit consisting of four elements. 

Suppose the action of $\lambda$ has an orbit consisting of three elements $x_1,x_2,x_3$.
We will show that $\lambda^6=1$. To see this, note first that (b) implies that 
for any $1\leq i,j \leq 3$, $x_ih\neq x_j$. It follows easily
that $x_1h,x_2h,x_3h$ form another orbit. Moreover, $x_3=x_1x_2$ or $x_3=x_1x_2h$.
In the former case, $x_1,x_2,x_3$ generate a rank two subgroup $Z^\prime$ of $Z$ 
invariant under $\lambda$, whose fixed-point set $\Sigma^{Z^\prime}$ is a $1$-sphere 
containing $\{p_{+},p_{-}\}$ because each $\Sigma^{x_i}$ is a $2$-sphere 
({\it cf.} \cite{MZ0}, Lemma 4.1). 
Moreover, $\lambda$ permutes the $2$-spheres $\Sigma^{x_1}$,$\Sigma^{x_2}$,
$\Sigma^{x_3}$. The $1$-sphere $\Sigma^{Z^\prime}$ is the intersection of $\Sigma^{x_1}$,
$\Sigma^{x_2}$, $\Sigma^{x_3}$, which is invariant under $\lambda$. Since $\lambda$ 
switches $p_{+}$ and $p_{-}$, it follows that $\lambda^2$ fixes $\Sigma^{Z^\prime}$. Now there 
are four $2$-spheres $\Sigma^{x_1}$,$\Sigma^{x_2}$,
$\Sigma^{x_3}$, and $\Sigma^{\lambda^2}$ intersecting at the $1$-sphere $\Sigma^{Z^\prime}$.
Examining the action of $\lambda$ on the normal bundle of $\Sigma^{Z^\prime}$, we see that
$\lambda^6$ acts as identity in a neighborhood of $\Sigma^{Z^\prime}$, which implies 
$\lambda^6=1$. In the latter case, apply the above
argument to the orbit $x_1h,x_2h,x_3h$, which generates a $\lambda$-invariant subgroup
of rank two. Hence the claim $\lambda^6=1$ follows. 

If the action of $\lambda$ has  an orbit consisting of two elements $x_1,x_2$, then we claim 
that there is a $z\in Z_0$ such that the set $\{z,zh\}$ is invariant under $\lambda$, which we 
have seen is a contradiction to the assumption $\lambda^2\neq 1$ as we argued in the
previous lemma. To see the claim, note that if $x_2=x_1h$, then we are done. If $x_2\neq x_1$,
then $x_1h,x_2h$ form another orbit. The complement $Z_0\setminus \{x_i,x_ih\}$ is a set
consisting of $z,zh$ which is invariant under $\lambda$. Hence the claim follows. 

Finally, if $\lambda$ fixes every element of $Z_0$, then every subset $\{x,xh\}$ is invariant 
under $\lambda$, which is a contradiction. This shows that either $\lambda^2=1$ 
or $\lambda^6=1$ for any $\lambda\in G\setminus H$.

We shall further rule out  the possibility that $\lambda$ has order $6$. To this end, 
we consider the centralizer $C_G(Z)$ of $Z$.
We first show that for any $g\in C_G(Z)$, $g^2=1$. To see this, note that the action of $g$ 
by conjugation leaves each $2$-sphere $\Sigma^z$, where $z\in Z_0$, invariant, which all contains 
$\{p_{+},p_{-}\}$. In particular, consider $2$-spheres $\Sigma^x$, $\Sigma^y$, $\Sigma^{xy}$, 
for some choices of $x,y\in Z_0$ with $xy\in Z_0$. The action of $g$ leaves each of them invariant, 
and so also leaves their intersection, which is a $1$-sphere, invariant. If $g\in H$, then $g$ fixes 
$\{p_{+},p_{-}\}$, and $g$ is either identity or a reflection on the $1$-sphere. 
It follows easily that $g^2$ fixes each of the $2$-spheres $\Sigma^x$, $\Sigma^y$, $\Sigma^{xy}$, 
hence $g^2=1$. If $g\in G\setminus H$, then $g$ switches $p_{+}$ and $p_{-}$, and the action of
$g$ on the $1$-sphere is by a rotation of order $2$. It follows also that $g^2$ fixes each of the 
$2$-spheres $\Sigma^x$, $\Sigma^y$, $\Sigma^{xy}$, and hence $g^2=1$. With this understood,
we see that $C_G(Z)$ is an elementary Abelian $2$-group. Now note that $S\subseteq C_G(Z)$ 
because
$Z$ lies in the center of $S$, so that $S=Z$ by the maximality of $Z$. On the other hand, since
each Sylow p-subgroup for odd $p$ lies in $C_G(Z)$, we see that $F=S=Z$ immediately. 

With the above preparation, we now rule out the possibility that $\lambda$ has order $6$.
The key observation is that in the proof of $\lambda^6=1$, we also showed that in that case
$\lambda^3$ acts on $Z=F$ trivially. But we know that $G/F$ is a subgroup of $\text{Out }(F)$. 
This implies that $\lambda^3\in F=Z\subset H$. 
Consequently, $\lambda\in H$ because $\lambda^2\in H$, which is a contradiction. 
This rules out  the possibility that $\lambda$ has order $6$, and consequently, 
$\lambda^2=1$ for any $\lambda\in G\setminus H$. The proof of Lemma \ref{lem:3} is
complete.

\end{proof}

Now we enter Stage 2 of the proof, where we assume that for every Sylow p-subgroup of $F$,
(i) is true. Furthermore, without loss of generality, we assume that there is a $\lambda \in
G\setminus H$ such that $\lambda^2\neq 1$. We remark that the Atiyah-Bott Theorem, through 
Lemma \ref{lem:1}, played a crucial role in this stage. 

\begin{prop}\label{prop:2}
Suppose (i) is true for every Sylow p-subgroup of $F$. Then $G$ is isomorphic 
to a subgroup of $\OO{4}$.
\end{prop}

It turns out that the Fitting subgroup $F$ must be cyclic, and 
there is a $\lambda \in G\setminus H$ of order $4$, such
that $H$ is the dihedral group generated by $F$ and $\lambda^2$. 

We begin with a group-theoretic observation.

\begin{obs}\label{obs:3}
Let $\xi$ be an element of order $p^n$ where $p$ is prime. Suppose $\lambda$ 
is an element such that 
$$
\lambda\xi\lambda^{-1}=\xi^k, \mbox{ where } k\neq 1 \pmod{p^n}
$$
and $\lambda^2$ commutes with $\xi$. Then
\begin{itemize}
\item if $p$ is odd, then $k=-1 \pmod{p^n}$;
\item if $p=2$, then either $k=-1 \pmod{p^n}$ or $k=\pm 1 \pmod{p^{n-1}}$. Moreover, in the
latter case, $n\geq 3$, and $\lambda\xi^2\lambda^{-1}=\xi^{\pm 2}$. 
\end{itemize}
\end{obs}

\begin{proof}
First, the assumptions $\lambda\xi\lambda^{-1}=\xi^k$ and $\lambda^2$ commutes with $\xi$
imply that $k^2=1 \pmod{p^n}$, or equivalently, $(k-1)(k+1)=0 \pmod{p^n}$. Since 
$k\neq 1 \pmod{p^n}$, one must have either $k=-1 \pmod{p^n}$, or else there exist $m_1,m_2$
satisfying $1\leq m_1,m_2\leq n-1$ and $m_1+m_2\geq n$, such that 
$$
k-1=p^{m_1} u_1, k+1=p^{m_2}u_2
$$
for some $u_1,u_2$ which are relatively prime to $p$. 

We continue the discussion assuming that $k\neq -1 \pmod{p^n}$. Consider first the case
where $m_1\leq m_2$. We write 
$$
2k=p^{m_1} u_1+p^{m_2}u_2=p^{m_1}(u_1+p^{m_2-m_1}u_2),
$$
which implies $p^{m_1}=2$. Consequently, $p=2$, $m_1=1$, and $m_2\geq n-m_1=n-1$
which implies $m_2=n-1$ and $k=-1 \pmod{p^{n-1}}$. Finally, observe that $m_2-m_1>0$ 
because $k$ is odd, which implies that $n=m_1+m_2\geq 3$. For the case where 
$m_2\leq m_1$, a similar argument shows that $k=1 \pmod{p^{n-1}}$, with $n\geq 3$.

\end{proof}

The following observation plays a key role in determining the Sylow 2-subgroup of $F$.

\begin{obs}\label{obs:4}
Let $h\in H$ be an element of order $2^n$ such that for any $k$ with $h^k\neq 1$, 
$\Sigma^{h^k}=\{p_{+},p_{-}\}$.
If the normalizer of the subgroup generated by $h$ contains a $\lambda\in G\setminus H$, then
$n=1$, i.e., $h^2=1$.
\end{obs}

\begin{proof}
We assume $n\geq 2$, and derive a contradiction.

Note that $\lambda\in G\setminus H$ has an even order, say $|\lambda|=k 2^{m+1}$ where
$k$ is odd and $m\geq 0$. Then $\lambda^k\in G\setminus H$ and is also contained in the
normalizer of the subgroup generated by $h$. After replacing $\lambda$ by $\lambda^k$,
we may assume $\lambda$ has order $2^{m+1}$, $m\geq 0$.

Set $\mu=\lambda^{2^m}$. Then $\mu^2=1$, and by Observation \ref{obs:3}, we have
$\mu h \mu^{-1}=h^k$, where $k$ satisfies either $k=\pm 1 \pmod{2^n}$ or 
$k=\pm 1 \pmod{2^{n-1}}$ with $n\geq 3$. It follows that $\mu\xi\mu^{-1}=\xi^{\pm 1}$
for $\xi=h$, or $h^2$, with the order of $\xi$ being a power of $2$ which is greater than $2$.

If $m=0$, then $\mu=\lambda\in G\setminus H$, which contradicts Lemma \ref{lem:1}.
Suppose $m>1$. We claim $\mu\xi\mu^{-1}=\xi$ must be true. Suppose to the
contrary that $\mu\xi\mu^{-1}=\xi^{-1}$. Then if we set $\tau=\lambda^{2^{m-1}}$, then
$\tau\xi\tau^{-1}=\xi^k$ for some $k$, where $k$ satisfies $k^2=-1 \pmod{|\xi|}$. But 
there is no solution for $k$ if $|\xi|$ is a power of $2$ greater than
$2$. Hence the claim follows. Now applying Observation \ref{obs:3} to $\tau$, we get 
$\tau\xi\tau^{-1}=\xi^{\pm 1}$ for some $\xi$, with the order of $\xi$ being a power of $2$ 
which is greater than $2$. By induction, we get a contradiction. 

\end{proof}

As an immediate corollary, we obtain 

\begin{lemma}\label{lem:4}
The Sylow $2$-subgroup $S$ of the Fitting subgroup $F$ is of order $2$.
\end{lemma}

\begin{proof}
Since the maximal elementary Abelian subgroup $Z$ of the center of $S$ has rank one
by assumption, $S$ is either cyclic, dihedral, quaternion, or generalized
quaternion ({\it cf.} Suzuki \cite{Su}, pp. 58-59). Now observe that the number of cyclic 
subgroups of
maximal order of $S$ is one except for the case when $S$ is quaternion, and when $S$
is quaternion the number is $3$. It follows that for a given $\lambda\in G\setminus H$, 
there exists a cyclic subgroup of maximal order, denoted by $C$, whose normalizer contains
$\lambda$. On the other hand, since $Z\subseteq C$ has fixed points $\{p_{+},p_{-}\}$, it follows
that every element of $C$ has fixed points $\{p_{+},p_{-}\}$. By Observation \ref{obs:4},
$C\cong \ZZ_2\cong Z$, from which it follows that $S=Z\cong \ZZ_2$. 

\end{proof}

Since we assume that (i) is true for every Sylow p-subgroup of $F$, in particular, 
the maximal elementary Abelian subgroup $Z$ of the center of Sylow $p$-subgroup 
for odd $p$ has rank one, the Sylow $p$-subgroups with $p$ odd are all cyclic ({\it cf.}
Suzuki \cite{Su}, pp. 58-59). Consequently, with Lemma \ref{lem:4}, we see that $F$ is cyclic.
Moreover, note that $F\subseteq H$. Since $G/F$ is a subgroup of $\text{Out } F=\text{Aut } F$
which is cyclic, we see that $H/F$ is cyclic, and we conclude that $H$ is metacyclic. 

\begin{lemma}\label{lem:5}
Suppose there is a $\lambda\in G\setminus H$ such that $\mu=\lambda^2\neq 1$,
and $F$ contains a Sylow p-subgroup of odd $p$.
Then $H$ is a dihedral group generated by $F$ and $\mu$, with $\mu^2=1$.
\end{lemma}

\begin{proof}
We will exploit the double cover $\rho: \SO{4}\rightarrow \SO{3}\times \SO{3}$, whose kernel
is $\pm I$. Note that $\ker \rho\cap H$ coincides with the Sylow 2-subgroup of $F$.
We denote by $K$ the subgroup of $H$ generated by $F$ and $\mu$, which is also metacyclic.
We denote by $K^\prime, H^\prime, F^\prime,\mu^\prime$ the image of $K,H,F,\mu$ under 
$\rho$ respectively. 
Note that the product of Sylow p-subgroups of $F$ with $p$ odd is mapped injectively onto
$F^\prime$. 

We will first show that $K^\prime$ is dihedral, which is generated by $F^\prime$ and $\mu^\prime$. 
To this end, we let $\pi_i$ be the projection from $\SO{3}\times \SO{3}$ onto its $i$-th factor, 
where $i=1,2$, and let $K_i=\pi_i(K^\prime)$, $F_i=\pi_i(F^\prime)$, and $\mu_i=\pi_i(\mu^\prime)$.
Since $K^\prime$ is metacyclic, so are both $K_1,K_2$. Since the only metacyclic subgroups
of $\SO{3}$ are either cyclic or dihedral, we see that $\mu_i^2$ lies in the centralizer of $F_i$.
Consequently, $(\mu^\prime)^2$ lies in the centralizer of $F^\prime$. It follows easily that
$\mu^2$ lies in the centralizer of each Sylow p-subgroup of odd order. Now observe that 
each Sylow p-subgroup of odd order acts on $\Sigma$ semifreely, it follows from 
Observation \ref{obs:3} and Lemma \ref{lem:1} (applied to $\lambda$) that $\mu \xi\mu^{-1}\neq \xi$ 
for each generator $\xi$ in the Sylow p-subgroups of odd order. Applying Observation \ref{obs:3} again
(this time to $\mu$), with the fact that $\mu^2$ lies in the centralizer of each Sylow p-subgroup of odd order,
we see that $\mu \xi\mu^{-1}=\xi^{-1}$ for each $\xi$ in the Sylow p-subgroups of odd order, 
which implies that each $K_i$ is either trivial, $\ZZ_2$, or dihedral. Consequently, 
$K^\prime$ is dihedral and $(\mu^\prime)^2=1$. Since $\mu$ has a $2$-dimensional 
fixed-point set, $\rho$ is injective on the subgroup generated by $\mu$. This implies that 
$\mu^2=1$.

Next we show that $H$ is generated by $F$ and $\mu$. 
To this end, let $H_i=\pi_i(H^\prime)$, $i=1,2$. Then
$H_i$ is either cyclic or dihedral, because $H^\prime$ is metacyclic. This implies that 
for any $g^\prime \in H^\prime$, which is the image of $g\in H$ under $\rho$, 
if we let $g_i=\pi_i(g^\prime)$, then $g_i^2$ lies in the centralizer of 
$F_i$, which implies as earlier that $g^2$ lies in the centralizer of each Sylow p-subgroup
of odd $p$. By Observation \ref{obs:3}, $g\xi g^{-1}=\xi^{\pm 1}$ on each Sylow p-subgroup 
of odd $p$.
We claim that either $g\xi g^{-1}=\xi^{-1}$ or $g\xi g^{-1}=\xi$ on $F$, which would imply in turn
that either $\mu g$ or $g$ lies in $F$, and we are done. 

The key point is the following: the real representation of each Sylow p-subgroup of odd $p$ on
the tangent spaces of $p_{+},p_{-}$ splits into a direct sum of two {\sf non--equivalent} real
$2$-dimensional representations. This would imply that each of the two $2$-dimensional representations is preserved by $g$, and furthermore, $g\xi g^{-1}=\xi^{-1}$ or $g\xi g^{-1}=\xi$ depends on whether $g$ is a reflection or a rotation on each of them. This would then imply that 
if $g\xi g^{-1}=\xi^{-1}$, respectively, $g\xi g^{-1}=\xi$ for one Sylow p-subgroup of odd $p$, then 
$g\xi g^{-1}=\xi^{-1}$, respectively, $g\xi g^{-1}=\xi$ for all other Sylow p-subgroups of odd $p$.
Hence the claim follows.

To see that the two $2$-dimensional representations are non-equivalent, suppose to the contrary
that they are equivalent. Then as argued in the proof of Lemma \ref{lem:1} (recall that
each Sylow p-subgroup of odd order acts on $\Sigma$ semifreely), the fact that 
$\lambda \xi \lambda^{-1}=\xi^k$ for some $k$ satisfying $k^2=-1 \pmod{p^n}$ implies that 
$\lambda$ is represented by a matrix $A=(A_{ij})$, such that $A_{ij}$ satisfies 
$A_{ij} H_\theta=H_{k\theta} A_{ij}$ for some $\theta$ where $\sin\theta\neq 0$. 
It follows that $A_{ij}^2$ satisfies 
$$
A_{ij}^2 H_\theta=H_{\theta}^{-1} A_{ij}^2,
$$
so that $A_{ij}^2$ has the form 
$$
\left (\begin{array}{ll}
a & a^\prime\\
a^\prime & -a\\
\end{array}
\right ).
$$
But this is a contradiction, as on the one hand, $\det A_{ij}^2=(\det A_{ij})^2\geq 0$, and on the other
hand, $\det A_{ij}^2=-a^2- (a^\prime)^2\leq 0$, which implies $A_{ij}=0$ for all $i,j$.

\end{proof}

\begin{proof} [Proof of Proposition \ref{prop:2}] If $F$ contains no Sylow p-subgroups of odd $p$, then
$F\cong \ZZ_2$, and it follows easily that $G\cong \ZZ_2$, which contradicts the assumption that there 
is a $\lambda\in G\setminus H$ such that $\lambda^2\neq 1$. Hence, $F$ must contain a Sylow 
p-subgroup of odd $p$. By Lemma \ref{lem:5},  $G$ is isomorphic to a subgroup of $\OO{4}$
by the following action: if $(z_1,z_2)$ is the complex representation of $F$ given by the
action on the tangent space at $p_{+}$ or $p_{-}$, then the action of $G$ is the one where the
action of $\lambda$ is given by $\lambda\cdot (z_1,z_2)=(\bar{z}_2, z_1)$.

\end{proof}

The proof of Theorem 1 is now complete. 

\section{Orientation reversing actions on $S^4$: local linearity and stable smoothability}

This section uses methods of surgery theory to study Conjecture 2. Since these methods are
of considerable complexity we start with general comments which we hope will shed some light
onto our considerations. 

Let $G$ be a finite group acting freely on some sphere $S^{n-1}$, 
where $n \geq 2$ and $(n-1)$ is odd.  Then it follows that $G$ has 
a free resolution of period $n$, and its Tate cohomology 
$\widehat{H}^*(G; \ZZ)$ is periodic of period $n$ ({\it cf.} \cite{Milnor}).

Now given a finite group $G$ of period $n$ one can ask if $G$ can
act freely on a finite $CW$-complex $X$ with $X \simeq S^{n-1}$.

It turns out that there is a finiteness obstruction $\sigma_n(G)$
(introduced by R. Swan in \cite{Sw}) for existence of such an action.  
This obstruction takes value in a certain quotient of 
$\widetilde{K_0}(\ZZ[G])$; {\it i.e.}, in $\widetilde{K_0}(\ZZ[G])/ 
T_G$ ({\it cf.} \cite{DM}).

Computation of these finiteness obstructions is in general a  very 
complicated and quite technical task.  For various groups it was 
carried successfully however by R. J. Milgram and I. Madsen ({\it cf.} 
\cite{Mg}, \cite{Ma}).  Quite extensive calculations were done for 
the class of groups $Q(2^k a, b, c)$, where $a,b,c$ are coprime 
integers, $k \geq 2$ and $Q(2^k a,b,c)$ is given by the semi-direct product
$$1 \to \ZZ_a \times \ZZ_b \times \ZZ_c
\to Q(2^k a, b, c) \to Q(2^k) \to 1$$
in which $Q(2^k)$ is the quaternionic 2-group ({\it cf.} \cite{DM}).  
In particular for $\pi = Q(8p,q,1)$, there are conditions on $p,q$ 
({\it cf.} \cite{Mg}, \cite{Ma}) which imply $\sigma_4(\pi)=0$.

Let $\widetilde{X} \simeq S^3$ be a finite complex with a free action 
of $\pi$ on $\widetilde{X}$.  Then $\widetilde{X}/\pi = X$ 
is a finite 3-dimensional Poincar\'{e} complex and hence it is equipped 
with a Spivak normal bundle ({\it i.e.}, a homotopy spherical fibration; 
{\it cf.} \cite{W1}, \cite{DM}).  Let $f: X \to BSG$ be the classifying 
map for this fibration ({\it cf.} \cite{DM}).

By considering Sylow $p$-subgroups $\pi_p$ of $\pi$ and using the fact 
that there are manifold models (linear space forms) for each lifting $\widetilde{X}_{(p)} 
\simeq S^3 / \pi_p$ one can conclude ({\it cf.} \cite{DM}, \cite{Ma}) that 
$f:X \to BSG$ lifts to $f:X \to BSO$ where $BSO$ is the classifying 
space for oriented bundles.  The existence of such a lifting leads to 
the existence of a normal map
$$f: (M^3, \nu_{M^3}) \to (X, \xi_X)$$
where $\nu$ is the stable normal bundle, and $\xi_X$ is the Spivak 
normal bundle.  Rather intricate and quite lengthy computations 
(see \cite{Be}, \cite{DM}, \cite{Ma}) show that the surgery 
obstruction $\lambda(f) \in L_3^h(\pi)$ is trivial for each of 
the pairs 
$$(p,q)=(3,313), (3,433), (3,601), (7,113), (5,461), (7,809), 
(11,1321), (17,103)~.$$
In dimension 3 this means ({\it cf.} \cite{JK}) that there is a 
manifold $M^3$ and a map $k:M^3 \to X$ which is a $\ZZ[\pi]$-homology 
equivalence (in other words: there is a free action of $\pi$ on 
some integral homology 3-sphere $\widetilde{M}^3$).

{\bf Note. } The functorial splitting of the rational group algebra $Q[Q(8p,q)]$ in \cite{BM},
pp. 452, leads to a corresponding splitting of the real representation ring $R[Q(8p,q)]$.
It follows that the only irreducible {\sf faithful} real representation has dimension $8$.
Consequently, $Q(8p,q)$ is not isomorphic to a subgroup of $\OO{n}$ for $n\leq 7$.

\vspace{1mm}

\noindent{\it Construction of actions in Theorem 2}. 
Let $\pi$ be the group $Q(8p, q)$, where $(p,q)$ is any of the pairs mentioned earlier.  
Let $E_X$ be the total space of a twisted $I$-bundle over $X$ (= the unit disk bundle of 
a real line bundle).

Now if $Q(8)$ is the quaternionic group given by
$$Q(8) = \{ x,y \vert x^4=1, x^2 = y^2, y x y^{-1} = x^{-1} \}$$
let $w: Q(8) \to \ZZ_2$ be a nontrivial orientation homomorphism 
given by
$$w(x) = +1, \; w(y) = -1.$$
It induces a corresponding homomorphism
$$w: Q(8p, q) \to \ZZ_2$$
which in turn yields a specific line bundle, and hence a specific 
choice of $E_X$.

Next, let $E_{M^3}$ be the pull back of $E_X$ by $k$.  Then there is a 
degree one map
$$h:(E_{M^3}, \partial) \to (E_X, \partial)$$
and $h$ is a $\ZZ[\pi]$-homology equivalence between manifold with boundary $E_{M^3}$ and $E_X$.  
It turns out that $\lambda(h) \in L_0^h(\pi, w)$ is trivial ({\it i.e.}, 
$L_0^h(\pi, w) \cong \Gamma_0(\FFF, w)$),
where $\Gamma_0(\FFF, w)$ are the homological surgery
obstruction groups as in \cite{CS}. Here $\FFF: \ZZ[\pi]\rightarrow \ZZ[\pi]$ is the identity 
homomorphism. 

Consequently we can replace the manifold $E_{M^3}$ by a manifold $(W^4, 
\partial)$ homotopy equivalent $(\rel \partial)$ to $(E_X, \partial)$.

In particular, $h_0 = h \vert_{\partial W^4} : \partial W^4 \to 
\partial E_X$ is a $\ZZ[\tau]$-homology equivalence where $\tau \subset 
\pi$ is a subgroup of index two; in fact, $\tau \cong Q(4pq)$.

Now applying the above argument to
$$h_0 \times \id : \partial W^4 \times I \to \partial E_X \times I$$
one obtains a manifold $(N^4, \partial)$ homotopy equivalent to $\partial E_X 
\times I$.  Let
$$\overline{N} = N^4 \cup_\partial N^4 \cup_\partial \ldots.$$
and form $M^4 = W^4 \cup_\partial \overline{N}$.  Then $M^4$ is a one-ended 
manifold with the universal covering $\widetilde{M^4}\approx_{top} S^3 \times\RR$
({\it cf.} \cite{F}).

As a consequence, the two-point compactification of $\widetilde{M^4}$ gives an action of $\pi$ 
on $S^4$; the induced group 
action on the two ``points at infinity'' is given by the homomorphism 
$w$ described above.  This action is fixed-point free (it is 
pseudo-free with two singular points having isotropy group $\tau$).

Consequently we have

\vspace{1mm}

{\bf Fact.}
{\it There is a finite group $\pi$ acting topologically and orientation reversingly on $S^4$ without 
fixed points such that $\pi$ is not isomorphic to a subgroup of $\OO{5}$.
}

\vspace{1mm}

{\bf Remarks.} Starting with the trivial $I$-bundle over $X$ (i.e., with $X\times I$) and repeating the above construction with the trivial orientation homomorphism $w:Q(8p, q) \to \ZZ_2$ one gets

\vspace{1mm}

{\bf Observation.} 
{\it There is a finite group $\pi$ acting topologically and orientation-preservingly on $S^4$ semifreely 
with two fixed points such that $\pi$ is not isomorphic to a subgroup of $\SO{5}$. 
}

\vspace{1mm}

{\bf Note.} Constructions like the ones above were known to the experts right after
Freedman's work on $4$-dimensional topological surgery. 

\vspace{1mm}

The above {\it Fact} gives a topological solution to Conjecture 2. 
It would be interesting to see if one can transform the topological action constructed above
into a locally linear one (and therefore provide counterexamples to Conjecture 2 in \cite{MZ}). 
It is not clear at this moment if this is the case, but we 
describe below a possible strategy for achieving this, namely: 

Let $(p,q)$ be a pair of primes such that $Q(8p,q)$ acts freely on 
a homology 3-sphere $\Sigma^3$.  Write $\pi = Q(8p, q)$ and $\tau = 
Q(4pq)$ where $\tau \subset \pi$ is a subgroup of index two.  

Let $h: \Sigma^3 / \pi \to X^3 / \pi$ be a $\ZZ[\pi]$-homology equivalence, 
where $X^3 / \pi$ is a finite Poincar\'{e} complex with 
$\widetilde{X^3} \simeq S^3$.

Consider the exact homology surgery sequence in dimension 3 
({\it e.g.}, see \cite{JK}):
$$ \xymatrix{
L_0^h(\pi) \ar[r]^-{\gamma} & {\mathcal S}_H(X^3 / \pi) \ar[r]^-{\eta} & 
\noin{X^3/\pi} \ar[r]^-{\Theta_3^h} & L_3^h(\pi) }
$$
The transfer to the 2-fold cover $X/\tau\to X/\pi$ gives a 
commutative diagram
$$ \xymatrix{
L_0^h(\tau) \ar[r]^-{\gamma} & {\mathcal S}_H(X^3 / \tau) \ar[r]^-{\eta} & 
\noin{X^3/\tau} \ar[r]^-{\Theta_3^h} & L_3^h(\tau)\\
L_0^h(\pi) \ar[r]^-{\gamma} \ar[u]_-{tr_*} & {\mathcal S}_H(X^3 / \pi) 
\ar[r]^-{\eta} \ar[u]_-{tr_*} & \noin{X^3/\pi} \ar[r]^-{\Theta_3^h} 
\ar[u]_-{tr_*} & L_3^h(\pi) \ar[u]_-{tr_*} }
$$
Since $X^3 / \tau \simeq S^3 / \tau$ and $S^3 / \tau$ is a manifold we 
have a base point in ${\mathcal S}_H(X^3 / \tau)$.

Let $tr_*(h) = \widetilde{h} : \Sigma^3 / \tau \to S^3 / \tau$ be a 
$\ZZ[\pi]$-homology equivalence and let $[\widetilde{h}] \in {\mathcal S}_H(X^3 /
\tau)$.  We claim that $\eta[\widetilde{h}] = 0$ in $\noin{X^3 / \tau}$.  
Indeed, $\noin{X^3 / \tau} \cong H_1(\tau; \ZZ_2) \cong \ZZ_2$ and since 
$\Theta_3^h$ is a monomorphism (see \cite{JK}), our claim follows.

This means that there is an element $\widetilde{x} \in L_0^h(\tau)$ with 
$[\widetilde{h}] = \gamma{\widetilde{x}}$.

\begin{hyp}\label{hyp:1}
Suppose that there is an element $x \in L_0^h(\pi)$ such that $tr_*(x) = 
\widetilde{x}$.
\end{hyp}

Assuming the above, we act on $[h] \in {\mathcal S}_H(X^3 / \pi)$ by $(-x)$. Then 
we get a new $\ZZ[\pi]$-homology equivalence:
$$\bar{h}: \Sigma'/\pi \to X^3 /\pi$$
with the lifting $\widetilde{\overline{h}}: \Sigma' / \tau \to S^3 / 
\tau$ such that $\widetilde{\overline{h}}$ is $H$-cobordant to 
$\id_{S^3/\tau}$.

Let $\overline{W^4}$ be such an $H$-cobordism
$$(\overline{W^4}; \Sigma'/\tau, S^3/ \tau) = (\overline{W^4}; 
\partial_0, \partial_1).$$
Now in our construction of the one-ended manifold $M^4$ we simply take
$$M^4 = W^4 \cup_{\partial_0} \overline{W^4} \cup_{\partial_1} 
(S^3/\tau) \times [0,\infty).$$
Obviously, $\widetilde{M^4} \approx_{top} S^3 \times \RR$ and the 
action of $\pi$ on the two-point compactification is locally linear.

To be even more specific, we should point out that local linearity of the above actions
is related to the following question, which may be interesting in its own right:

\vspace{1mm}

{\it Suppose that $G$ is a finite group which acts freely on some integral homology 
$3$-sphere $\Sigma^3$, let $H\subset G$ be a subgroup which acts freely on $S^3$, 
and assume that there is a twisted $\ZZ[H]$-homology equivalence from $\Sigma^3/H$
to some spherical space form $L(\Sigma,H)$. Can one find a free $G$-action on some
homology $3$-sphere $\Sigma^\prime$ such that the Atiyah-Singer $\rho$-invariants 
(cf. \cite{ASIII}) of $(\Sigma^\prime,H)$ and $L(\Sigma,H)$ are equal?
}

\vspace{1mm}

If $G$ acts freely and smoothly on some closed oriented $3$-manifold $M^3$, then
the invariant $\rho(M^3,G)$ of \cite{ASIII} is a complex valued function on $G\setminus \{1\}$
which is a rational number times the character of a virtual real representation of $G$.

Analogous questions in higher dimensions are studied in \cite{GT}. 

Assume now that $G=Q(8p,q)$ where $p$ and $q$ are distinct odd primes, and take 
$H$ to be the index $2$ subgroups $Q(4pq)$, which is a generalized quaternion group,
as before, choose $(p,q)$ so that $G$ acts freely on some integral homology $3$-sphere. 
For some choices of $p$ and $q$ --- including $(17, 103)$, $(3,313)$ and $(3,433)$ 
--- an affirmative answer to the $\rho$-invariant question would imply that one can choose 
the group actions constructed above to be locally linear; this uses computational techniques 
from \cite{W2}, \cite{O} and \cite{Ma} together with class number computations for certain 
rings of algebraic integers.

\vspace{2mm}

\noindent{\it Proof of Theorem 2.}
We begin with the following commutative diagram, in which all vertical
and horizontal lines are exact fibration sequences:
$$
\CD
Top/PL @>>> G/PL @ >\varphi'>> G/Top @ >{\overline {k}}>> K(\ZZ_2,4)\\
@V = VV @ V i_1 VV @V i_2 VV @V = VV\\
Top/PL @>>> BPL @ >\varphi >> BTop @ >{k} >> K(\ZZ_2,4)\\
@ VVV @VVV @VVV @VVV\\
\cdot @ >>> BG @ > = >> BG @ >>> \cdot
\endCD
$$
In this diagram $k$ denotes the universal triangulation obstruction
of \cite{KSb}.  Since all spaces in the diagram are $H$-spaces and
all morphisms are $H$-maps ({\it cf. } \cite{KSb}), it follows that
we have the following exact sequence of abelian groups, in which
$M/\partial$ is an abbreviation for $M/\partial M$:
$$ \xymatrix{
[M/\partial, Top/PL] \ar[r] & [M/\partial,G/PL] \ar[r]^-{\varphi'_*} &
[M/\partial, G/Top] \ar[r]^{{\overline {k}}_*} & H^4(M/\partial;\ZZ_2) }
$$
Observe that the group at the extreme left is isomorphic to
$H^3(M/\partial;\ZZ_2)$ and the group at the extreme right is
isomorphic to $\ZZ_2$.

Now the 4-stage Postnikov approximation to $G/Top$ is 
$K(\ZZ,2)\times K(\ZZ,4)$, and the corresponding approximation to
$G/PL$ has classification invariant $\beta \tinycirc \textrm{Sq}^2\in
H^5(K(\ZZ_2,2);\ZZ)$, where $\beta:H^4(-;\ZZ_2)\to
H^5(-;\ZZ)$ is the Bockstein operation for the short exact sequence
$$0~\to~\ZZ~\to~\ZZ~\to~\ZZ_2~\to 0~.$$
The results of \cite{KSb} then imply that the group
\medskip

\centerline{$\mathrm{coker}~\varphi'_*:[M/\partial,G/PL]\to [M/\partial,G/Top]~~\cong$}
\centerline{$\mathrm{image}~{\overline{k}}_*:[M/\partial,G/Top]\to H^4(M/\partial;\ZZ_2)~\cong~\ZZ_2$}

\noindent is equal to the subgroup of $H^4(M/\partial;\ZZ_2)$ given by
$$\mathrm{red}_2\,H^4(M/\partial;\ZZ)~+~\mathrm{Sq}^2 H^2(M/\partial;\ZZ_2)$$
where $\mathrm{red}_2$ denotes the mod 2 reduction operation from 
$\ZZ$ to $\ZZ_2$ coefficients.  It follows that the cokernel of
$\varphi'_*$ is nontrivial and the map ${\overline {k}}_*$ is onto.
Furthermore, if $c:M/\partial \to S^4$ is a map with mod 2 degree equal
to 1 which collapses the closed complement of a coordinate disk
$D^4\cong D\subset \mathrm{Int}(M)$ to a point, then it follows that
the composite
$$ \xymatrix{
\ZZ\cong\pi_4(G/Top) \ar[r]^-{c^*} & 
[M/\partial, G/Top] \ar[r]^{{\overline {k}}_*} & 
H^4(M/\partial;\ZZ_2)\cong\ZZ_2 }
$$
is onto.

It is worthwhile to point out that $c^*(1)$ is the normal invariant of $(M\# |E_8|,\partial)\to (M,\partial)$ and this map is normally bordant to a homotopy equivalence. Here $|E_8|$ is the manifold constructed by M. Freedman in \cite{F}.

\begin{claima}
If $w: Q(8p, q) \to \ZZ_2$ is the previously described 
homomorphism, then the surgery obstruction
map $\sigma: [M/\partial, G/Top] \to L_0^h(Q(8p,q),w)$ is
trivial on the image of $c^*$.  
\end{claima}

Before proving the claim, we shall explain how it implies Theorem 2.  
There are two cases depending upon whether or not
the relative Kirby-Siebenmann obstruction for the pair $(M,\partial M)$
is trivial or nontrivial (note that the boundary is uniquely smoothable 
and in fact is a spherical space form).  In both cases, one crucial
point is the fact that if $h:(N,\partial N)\to (M,\partial M)$ is a
homotopy equivalence which is a diffeomorphism on the boundary and
the normal invariant of $h$ is the nontrivial element $x$ in the image 
of $c^*$, then the relative Kirby-Siebenmann invariants of 
$(N,\partial N)$ and $(M,\partial M)$ are opposites because their
difference is the nontrivial class ${\overline {k}}_*\tinycirc c^*(x)$.

Suppose first that the relative Kirby-Siebenmann invariant of 
$(M,\partial M)$ is nontrivial, and let $h$ be as above.  Then by the
last sentence of the preceding paragraph the relative invariant of
$(N,\partial N)$ is trivial.  The boundary of $N$ is a spherical
space form for the group $K = \mathrm{Kernel}(w)$; let $V$ be
the associated 4-dimensional linear representation of that group,
and let $D(V)$ denote its unit disk.  If $\Sigma^4$ is formed
by equivariantly identifying the boundaries of ${\widetilde {N}}$
and $Q(8p,q)\times_K D(V)$, then it follows that $\Sigma^4$ is
a locally linear pseudofree $Q(8p,q)$-manifold, and the direct sum of
its topological equivariant tangent bundle in the sense of \cite{LR} 
with a trivial 1-dimensional bundle comes from an equivariant
vector bundle.  Therefore the results of \cite{LR} imply that
$\Sigma^4\times\RR^k$ is equivariantly smoothable for all $k\geq 1$.

In contrast, if the relative Kirby-Siebenmann invariant  of
$(M,\partial M)$ is trivial and $h$ is as above,  then we may
again construct $\Sigma^4$ as in the preceding paragraph; if
the direct sum of its topological equivariant tangent bundle 
in the sense of \cite{LR} with a trivial $k$-dimensional bundle 
does {\sf NOT} come from an equivariant vector bundle for all 
$k\geq 1$, then the results of \cite{LR} and \cite{LT} will imply that no
product of the form $\Sigma^4\times\RR^k$ can be equivariantly 
smoothable.  We need to show this tangent bundle assertion
holds if the Kirby-Siebenmann invariant of $(M,\partial M)$ is 
trivial, which means that the corresponding invariant for
$(N,\partial N)$ is not trivial.  We shall do this by assuming
the latter holds but $\Sigma^4\times\RR^k$ is nevertheless 
equivariantly smoothable and deriving a contradiction.

Suppose that a product as above is smoothable; without loss of 
generality we may assume that $k\geq 5$.  Since topological and linear
equivalence are the same for semifree representations ({\it e.g.}, see \cite{HP},
Corollary 3.10),  it follows that an invariant tubular 
neighborhood $E$ of the singular set must be equivariantly diffeomorphic to 
$$\RR^k\times \left(\,Q(8p,q)\times_K D(V)\,\right)$$ 
where $K$ and $V$ are defined as above.  Let ${\widetilde{W}}$ be 
the closure of the complement of such an invariant tubular neighborhood,
and let $W$ be its (smooth) orbit space.  It follows immediately that
the relative Kirby-Siebenmann obstruction for $(W, \partial W)$ is 
trivial.  

On the other hand, we claim that this obstruction
is equal to the relative Kirby-Siebenmann obstruction for 
$(N,\partial N)\times \RR^k$.  To see this, first observe that we can 
choose the invariant smooth tubular neighborhood $E$
to be so small that it lies in the interior of 
$$\left(Q(8p,q)\times_K D(V)\,\right)
\times\RR^k \subset \Sigma^4\times\RR^k$$
(viewing the left side as a subset of the right via the construction
of $\Sigma^4\times\RR^k$).  Now the region between $\partial E$ 
and $Q(8p,q)\times_K S(V)\times\RR^k$ is a proper equivariant 
$s$-cobordism ({\it e.g.}, this follows from \cite{Sieb}), and this fact
allows us to extend the linear structure on the topological 
equivariant tangent bundle of $E$ (given by smoothability) to a
linear structure on the equivariant topological tangent bundle of
$\left(Q(8p,q)\times_K D(V)\,\right)
\times\RR^k$.  It is natural to ask whether
or not this linear structure coincides with the standard one
given by product constructions on the equivariant tangent bundle
of $D(V)$, at least up to homotopy; fortunately, this follows 
directly from the $K$-equivariant contractibility of 
$D(V)\times \RR^k$ and the fact that the local representation
of $K$ at its fixed points is $V\times \RR^k$ in all cases. On the 
orbit space level this compatibility of linear equivariant tangent 
bundle structures implies that 
the relative Kirby-Siebenmann invariants for $(W,\partial W)$ and
$(N,\partial N)\times\RR^k$ must be equal, and as noted in an earlier 
paragraph this contradiction implies that no product of the form 
$\Sigma^4\times\RR^k$ can be equivariantly smoothable.

To summarize, regardless of whether or not the original action on
$S^4$ is stably smoothable, we have constructed a second action 
which has the opposite property.

\begin{proof}[Proof of Claim.]
Let $D\subset \mathrm{Int}(M)$ be an embedded closed 4-disk which 
extends to a slightly larger open 4-disk, so that $M-\mathrm{Int}(D)$
is a manifold with boundary.  One can construct a relative surgery 
problem ({\it i.e.}, a homeomorphism on the boundary) with target $M$
and surgery obstruction $c^*(x)\in [M/\partial,G/Top]$ as follows:
Start with a relative topological surgery problem $(f,\Phi)$ with
target $D^4$ which corresponds to a generator 
$x\in \pi_4(G/Top)\cong\ZZ$.  Now construct a relative surgery
problem $(f',\Phi')$ with target $M$ from $(f,\Phi)$ and the identity 
on $M-\mathrm{Int}(D)$ by identifying the homeomorphism to 
$\partial D^4$ in the first problem with the identity on
$\partial D$ in the second.  A direct analysis of this construction
shows that the resulting surgery problem represents $c^*(x)$ and
that the surgery obstruction is given by $\sigma(f') = i_*\sigma(f)$,
where 
$i_*:L_0(1)~\longrightarrow~L_0^s(Q(8p,q),w)$ is induced by
the canonical homomorphism $\ZZ\to\ZZ[Q(8p,q)]$ of rings with unit.
Therefore the proof of the Claim reduces to showing that $i_*$
is trivial.  We can reduce things further as follows:  The 
Sylow 2-subgroup of $Q(8p,q)$ is the quaternionic group $Q_8$, so
it suffices to prove that $L_0(1)\to L_0^s(Q_8,w')$ is trivial,
where $w'$ is the restriction of $w$.  Since the forgetful
map from $L_0^s(Q_8,w')$ to $L_0^h(Q_8,w')$ is an 
isomorphism (because ${Wh}(Q_8) = 0$), it actually suffices
to check that the map  $j_*:L_0(1)\to L_0^h(Q_8,w')$ is trivial.

The results of \cite{W2} and \cite{CS2} show that $L_0^h(Q_8,w')$ is
isomorphic to  $\ZZ_2\oplus\ZZ_2$.  Furthermore, the results of
\cite{H0}, pp. 122, show that one summand is detected by the composite
$$ \xymatrix{
L_0^h(Q_8,w') \ar[r]^-{f} & L_0^p(Q_8,w') \ar[r]^-{w_*} & 
L_0(\ZZ_2,-)\cong\ZZ_2 }
$$
where $f$ is the forgetful map from homotopy to projective Wall groups
(see \cite{PR}) and $w_*$ is induced by the group homomorphism
$w':Q_8\to \ZZ_2$; we can suppress the superscipt on the third
Wall group in the display because ${Wh}(\ZZ_2)$ and ${\widetilde
{K_0}}(\ZZ[\ZZ_2])$ are both trivial.  In fact, the results of \cite{H0}
show that $w_*$ is an isomorphism and the map $f$ is onto; one
of the $\ZZ_2$ summands in $L_0^h(Q_8,w')$ maps onto the generator
of $L_0^p(Q_8,w')$, and the other is given by the kernel of $f$.
By the Rothenberg exact sequence relating $L^h$ and $L^p$ (see 
\cite{PR}), this kernel is the image of the connecting homomorphism 
$$\delta:\ZZ_2\cong
H^*\left(\ZZ_2;{\widetilde{K_0}}(\ZZ[Q_8])\,\right)\to
L_0^h(Q_8,w')$$
which is injective.  Furthermore, if $b:\ZZ_4\to Q_8$ denotes the
inclusion of the kernel of $w'$, then the results described on 
page 633 of \cite{CS2}
imply that the image of $\delta$ maps nontrivially under the 
homomorphism
$$L_0^h(Q_8,w')\to L_0^h(Q_8,\ZZ_4;w')$$
in the long exact sequence of $L^h$ groups associated to the inclusion
homomorphism $b$.  Therefore, the triviality of the mapping $j_*$
is equivalent to the triviality of both $L_0(1)\to L_0(\ZZ_2,-)$
and the composite
$$L_0(1)\longrightarrow L_0^h(Q_8,w')\longrightarrow 
L_0^h(Q_8,\ZZ_4;w')~.$$

The first of these maps is trivial by results of Wall (see \cite{W1},
pp. 173, and \cite{W66}, (4.13, Complement), for details).  To prove
the triviality of the second map, observe that $j_*$ factors as
a composite
$$ \xymatrix{
L_0(1) \ar[r] & L_0(\ZZ_4) \ar[r]^-{b_*} & L_0^h(Q_8,w') }
$$
and by \cite{CS2} there is an exact sequence
$$ \xymatrix{
& L_0(\ZZ_4) \ar[r]^-{b_*} & L_0^h(Q_8,w') \ar[r] & 
L_0^h(Q_8,\ZZ_4;w') \ar[r] & 0~. }
$$
The combination of these statements imply that $L_0(1)\to
L_0^h(Q_8,\ZZ_4;w')$ is trival, and as noted above this
completes the proof of the claim.
\end{proof}

\vspace{3mm}

{\sl Acknowledgment. } The first named author is partially supported by
NSF grant DMS-1065784 and the second named author by Simons Foundation grant 
No. 281810. Part of this work was initiated when the first and the second 
named authors attended the FRG workshop ``The Topology and Invariants of Smooth 
$4$-manifolds" in Miami, March 12-16, 2012. We thank the organizer Nikolai Saveliev
for his hospitality and are grateful for the financial support through the NSF FRG grant 
DMS-1065905 and the University of Miami. All authors would like to thank Rafal 
Komendarczyk for the helping hand with the typesetting. Finally, the authors would 
like to thank the referee for the many suggestions which have improved considerably 
the exposition of the paper. In particular, 
we would like to thank the referee for directing our attention to Problem 2.

\vskip.25in
\begin{tabular}{lll}
Weimin Chen&&S{\l}awomir Kwasik \\
Department of Math.\& Stat.&&Department of Math. \\
University of Massachusetts &&Tulane University \\
Amherst, MA 01003 &&New Orleans, LA 70118 \\
{\tt wchen@math.umass.edu}&&{\tt kwasik@tulane.edu} \\
\end{tabular}
\begin{tabular}{lll}
Reinhard Schultz \\ 
Department of Math.\\
University of California \\
Riverside, CA 92521 \\
{\tt schultz@math.ucr.edu}
\end{tabular}
\end{document}